\documentclass[12pt]{amsart}
\usepackage{amssymb,a4wide}
\usepackage[usenames]{color}

\newcommand{\D}{{\mathbb D}}
\newcommand{\C}{{\mathbb C}}

\newtheorem{Theorem}{Theorem}
\newtheorem{Proposition}{Proposition}
\newtheorem{Lemma}{Lemma}

\title[Spectral properties of the canonical solution to $\bar\partial$]
{Some spectral properties of the canonical solution operator to $\bar\partial$
on weighted Fock spaces}

\author{Olivia Constantin and Joaquim  Ortega-Cerd\`a}
\date{\today}
\thanks{The second author is supported by the project MTM2008-05561-C02-01 and
the CIRIT grant 2005SGR00611}

\address{ Olivia Constantin,
Faculty of Mathematics,
University of Vienna,
Norbergstrasse 15, 
1090 Vienna, 
Austria} \email {olivia.constantin@univie.ac.at}

\address{Joaquim  Ortega-Cerd\`a, 
Departament de Matem\`atica Aplicada i An\`alisi, 
Universitat de Bar\-ce\-lo\-na, 
Gran Via 585, 08007 
Barcelona, Spain}\email{jortega@ub.edu}

\begin{document}
\maketitle

\begin{abstract}
We characterize the Schatten class membership of the canonical solution operator
to $\bar\partial$ acting on $L^2(e^{-2\phi})$, where $\phi$ is a subharmonic function
with $\Delta\phi$ a doubling measure.  The obtained characterization is in terms 
of $\Delta\phi$. As part of our approach, we study Hankel operators with anti-analytic
symbols acting on the corresponding Fock space of entire functions in $L^2(e^{-2\phi})$.
\end{abstract}


\section{Introduction}

For a (nonharmonic) subharmonic function $\phi$ on $\C$ having the property that
$\Delta\phi$ is a doubling measure, the generalized Fock space ${\mathcal
F}_\phi^2$ is defined by 
$$
{\mathcal F}_\phi^2=\{f\in{\mathcal H}(\C): \|f\|^2_{{\mathcal
F}_\phi^2}=\int_\C |f(z)|^2 e^{-2\phi(z)}\, dm(z)<\infty\},
$$
where $dm(z)$ denotes the Lebesgue measure on $\C$.  We let $\mu=\Delta\phi$ and
denote by  $\rho(z)$ the positive radius for which we have
$\mu(D(z,\rho(z)))=1,\, z\in\C$. The function $\rho^{-2}$ can be regarded as a
regularized version of $\Delta \phi$ (see \cite{christ, mmoc}). We consider the
canonical solution operator $N$ to $\bar \partial$  given by
$$\bar\partial Nf=f
\hbox{\ \ \ \  and\ \ \ \ } Nf \hbox{ is of minimal norm in }L^2(e^{-2\phi}),
$$ 
or, equivalently
$$
\bar\partial Nf=f \ \  \hbox{ and }\ \ Nf\perp \mathcal{F}^2_\phi.
$$
The boundedness and the compactness of $N$ acting on various weighted
$L^2$-spaces have been extensively studied in one or several variables (see
\cite{fustraube,haslinger1,haslhelffer,haslingerlamel}). Concerning the Schatten
class membership of this operator, it was first shown in \cite{haslinger1} that
for the particular choice $\phi(z)=|z|^m$, $N$ fails to be Hilbert-Schmidt, and
a more involved study was pursued in \cite{haslingerlamel} in the context of
several complex variables, where the authors obtain necessary and sufficient
conditions for the canonical solution operator to $\bar\partial$ to belong to
the Schatten class $\mathcal{S}^p,\,p>0$, when restricted to $(0,1)$-forms with
holomorphic coefficients in $L^2(\mu)$, for measures $\mu$ with the property
that the monomials form an orthogonal family in $L^2(\mu)$. Some particular
cases of these results were previously obtained in \cite{loverayoussfi}.  

In the present paper we are interested in the setting of subharmonic functions
$\phi$ with $\Delta \phi$ a doubling measure. For this type of weights, it was
proven in \cite{jj} that $N$ is compact from $L^2(e^{-2\phi})$ to itself if and
only if $\rho(z)\rightarrow 0$ as $|z|\rightarrow\infty$. We continue the
investigation in \cite{jj} by characterizing the Schatten class membership of
$N$.  We find that $N$ fails to be Hilbert-Schmidt and that $N$ belongs to the
Schatten class $\mathcal{S}^p$  with $p>2$ if and only if the following holds
\begin{equation}\label{introcrit}
\int_\C \rho^{p-2}(z)\, dA(z)<\infty.
\end{equation}
We start our approach by noticing that the restriction of $N$ to
$\mathcal{F}^2_\phi$ is actually a (big) Hankel operator with symbol $\bar z$.
This observation leads us to a study of these properties for Hankel operators on
$\mathcal{F}^2_\phi$ with anti-analytic symbols. We would like to point out that
Lin and Rochberg \cite{linrochberg, linrochberg1} considered these problems for
Hankel operators with symbols in $L^2(\C)$ for a certain class of subharmonic
functions $\phi$.  The case of anti-analytic symbols  was investigated in
\cite{schneiderknirsch,youssfi, schneider} for $\phi(z)=|z|^m,\, m>0$, and it
was shown that a Hankel operator $H_{\bar g}$ belongs to $\mathcal{S}^p$ if and
only if the symbol $g$ is a polynomial of degree smaller than $m(p-2)/(2p)$. For
subharmonic functions $\phi$ with $\Delta \phi$ a doubling measure, we find that
$H_{\bar g}$ fails to be Hilbert-Schmidt unless $g$ is constant, and $H_{\bar
g}\in\mathcal{S}^p$ for $p>2$, if and only if its symbol satisfies
$$
\int_\C |g'(z)|^p\rho^{p-2}(z)\, dA(z)<\infty,
$$
that is, $g$ is a polynomial whose degree depends on the order of decay of
$\rho$.

Finally, using a result by Russo \cite{russo}  together with the pointwise
estimates obtained in \cite{jj} for the kernel of the canonical solution
operator $N$, we show that the condition (\ref{introcrit}) is actually
sufficient for $N$ to belong to $\mathcal{S}^p$ with $p>2$, even when defined on
the whole of $L^2(e^{-2\phi})$.

\section{Preliminaries}
In this section we gather a few definitions and some  known estimates that will
be used in our further considerations. We start with some facts about doubling
measures. A nonnegative Borel measure $\mu$ is called {\it doubling} if there
exists $C>0$ such that
$$
\mu(D(z,2r))\le C\mu(D(z,r)),
$$
for all $z\in\C$ and $r>0$. The smallest constant in the previous inequality is
called the doubling constant for $\mu$. 

\begin{Lemma}\label{lc}(\cite[Lemma 2.1]{christ}) Let $\mu$ be a doubling
measure on $\C$.  There exists a constant $\gamma>0$ such that for any discs $D,
D'$ with respective radius $r>r'$ and with $D\cap D' \neq \emptyset$  the
following holds
$$
\Bigl( \frac{\mu(D)}{\mu(D')}\Bigr)^\gamma \lesssim \frac{r}{r'}\lesssim
\Bigl(\frac{\mu(D)}{\mu(D')}\Bigr)^{1/\gamma}. 
$$
\end{Lemma}
From now on we shall assume that $\phi$ is a subharmonic function on $\C$ such
that $\Delta\phi$ is a doubling measure.We denote $D^r(z)=D(z,r\rho(z))$ and for
$r=1$ we simply write $D(z)$ instead of $D^1(z)$. The function $\rho$ has at
most polynomial growth/decay (see \cite[Remark 1]{mmoc}): there exist constants
$C,\beta,\gamma>0$ such that
\begin{equation}\label{decayrho}
C^{-1}\frac{1}{|z|^\gamma}\le \rho(z)\le C |z|^\beta,  \hbox{ \ \ for\ \  }
|z|>1.
\end{equation}
As an immediate consequence of Lemma \ref{lc} one obtains
\begin{Lemma}\label{comparable}\cite{jj} For any $r>0$ there exists $c>0$
depending only on $r$ and the doubling constant for $\Delta \phi$ such that
$$
c^{-1} \rho(\zeta)\le \rho(z)\le c\,\rho(\zeta) \hbox{\ \  for \ \ } \zeta\in
D^r(z). 
$$
\end{Lemma}

\noindent We also have
\begin{Lemma}\label{distance1}\cite[p.~205]{christ} If $\zeta\notin D(z)$, then
$$
\frac{\rho(z)}{\rho(\zeta)}\lesssim\Bigl(\frac{|z-\zeta|}{\rho(\zeta)}\Bigr)^{
1-\delta}
$$
for some $\delta\in(0,1)$ depending only on the doubling constant for $\Delta
\phi$.
\end{Lemma}

\noindent For $z,\zeta \in \C$, the distance $d_\phi$ induced by the metric
$\rho^{-2}(z)dz\otimes  d\bar z$ is given by
$$
d_\phi(z,\zeta)=\inf_\gamma \int_0^1 |\gamma'(t)|\frac{dt}{\rho(\gamma(t))},
$$
where $\gamma$ runs over the piecewise $\mathcal{C}^1$ curves $\gamma :
[0,1]\rightarrow \C$ with $\gamma(0)=z$ and $\gamma(1)=\zeta$. We observe now
that the metric $\rho^{-2}(z)dz\otimes  d\bar z$ is comparable to the 
Bergman metric: it is well known, see \cite{bergman} that the Bergman metric 
$B(\frac{\partial}{\partial z},z)$ at the point $z$ is given by the 
solution to the extremal problem
\[
 B\Bigl(\frac{\partial}{\partial z},z\Bigr)=\frac{\sup\{|f'(z)|:\ f\in
{\mathcal F}_\phi^2,\ f(z)=0; \|f\|_{{\mathcal F}_\phi^2}=1\}}{\sqrt{{K(z,z)}}}.
\]
where $K(z,\zeta)$ is the Bergman kernel for $\mathcal{F}^2_\phi$. In
\cite[Lemma 20]{mmoc} it is proved that for all $f\in {\mathcal F}_\phi^2$ with
$f(z)=0$ we have $|f'(z)|\lesssim \frac{e^{\phi(z)}}{\rho^2(z)} \|f\|_{{\mathcal
F}_\phi^2}$, thus $B(\frac{\partial}{\partial z},z)\lesssim 1/\rho(z)$. The
other inequality follows taking as $f(\zeta)=C_z(\zeta-z)K(\zeta,z)$
where $C_z$ is taken in such a way that $\|f\|_{{\mathcal F}_\phi^2}=1$. In view
of the estimates of the Bergman kernel stated below it follows that
$B(\frac{\partial}{\partial z},z)\sim 1/\rho(z)$

The following estimates for the Bergman distance $d_\phi$ hold: 
\begin{Lemma}\label{distance2}\cite[Lemma 4]{mmoc}
There exists $\delta\in(0,1)$ such that for every $r>0$ there exists $C_r>0$
such that
$$
C_r^{-1} \frac{|z-\zeta|}{\rho(z)}\le d_\phi(z,\zeta)\le C_r
\frac{|z-\zeta|}{\rho(z)}, \quad \text{for }\zeta\in D^r(z), 
$$
and
$$
C_r^{-1} \Bigl(\frac{|z-\zeta|}{\rho(z)}\Bigr)^\delta\le d_\phi(z,\zeta)\le C_r
\Bigl(\frac{|z-\zeta|}{\rho(z)}\Bigr)^{2-\delta}, \quad \text{for }\zeta\in
D^r(z)^c.
$$ 
\end{Lemma}

\noindent The next result shows that we can  replace the weight $\phi$ by a
regular weight $\tilde\phi$ equivalent to it.
\begin{Proposition}\label{regularization}\cite[Theorem 14]{mmoc}
Let $\phi$ be a subharmonic function with $\mu=\Delta \phi$ doubling. There
exists
$\tilde \phi\in C^\infty(\C)$ subharmonic such that $|\phi-\tilde\phi|\le  c$
with $\Delta \tilde \phi$ doubling 
and 
$$
\Delta \tilde\phi \sim\frac{1}{\rho^2_{\tilde\phi}}\sim\frac{1}{\rho^2_\phi}.
$$
\end{Proposition}
\noindent  We also need the estimates
\begin{Lemma}\label{estimate}\cite{jj}
Let $\phi$ be a subharmonic function with $\mu=\Delta \phi$ doubling. Then for
any $\varepsilon>0$ and $k\ge 0$
$$
\int_\C \frac{|z-\zeta |^k}{\exp d_\phi(z,\zeta)^\varepsilon}\, d\mu(z)\le c
\,\rho^k(\zeta),
$$
where $c>0$ is a constant depending only on $k,\varepsilon$ and on the doubling
constant for $\mu$.
\end{Lemma}

\begin{Theorem}\label{kernelestimate}\cite{jj}
Let $K(z,\zeta)$ be the Bergman kernel for $\mathcal{F}^2_\phi$. There
exist positive constants $c$ and $\varepsilon$ (depending only on the doubling
constant 
for $\Delta\phi$) such that for any $z,\zeta\in \C$
$$
 |K(z,\zeta)|\le
c\frac{1}{\rho(z)\rho(\zeta)}\frac{e^{\phi(z)+\phi(\zeta)}}{\exp
d^\varepsilon_\phi(z,\zeta)}.
$$
\end{Theorem}

\begin{Lemma}\label{close}\cite{jj}
There exists $\alpha>0$ such that
$$
|K(z,\zeta)|\sim K(z,z)^{1/2}K(\zeta,\zeta)^{1/2}\sim
\frac{e^{\phi(z)+\phi(\zeta)}}{\rho(z)
\rho(\zeta)},\quad \hbox{ if } \ |z-\zeta|<\alpha \rho(z).
$$
\end{Lemma}

\noindent On the diagonal we have 
\begin{equation}\label{diagonalestimate}
K(z,z)\sim \frac{e^{2\phi(z)}}{\rho^2(z)},\quad z\in\C.
\end{equation}

\noindent For $\lambda\in\D$, we denote by $k_\lambda$ the normalized
reproducing kernel 
of $\mathcal{F}^2_\phi$, i.e.
$$
k_\lambda(z)=\frac{K(z,\lambda)}{K(\lambda,\lambda)^{1/2}},\quad z,\lambda\in\C.
$$ 
Finally, let us recall that a compact operator $T$ acting on a Hilbert space 
belongs to the Schatten class $\mathcal{S}^p$ if the sequence of eigenvalues 
of $(T^*T)^{1/2}$  belongs to $l^p$.

\section{Hankel operators on $\mathcal{F}^2_\phi$}

\noindent As already mentioned in the introduction, the canonical solution 
operator $N$ to $\bar \partial$ is defined on $L^2(e^{-2\phi})$ by
$$\bar\partial Nf=f
\hbox{\ \ \ \  and\ \ \ \ } Nf\perp \mathcal{F}^2_\phi.
$$ 
Let us now consider the restriction of $N$ to $\mathcal{F}^2_\phi$. Notice that
if $f\in \mathcal{F}^2_\phi$ and $\bar z f\in L^2(e^{-2\phi})$, then 
\begin{equation}\label{hankeldbar}
Nf=(I-P)(\bar z f),
\end{equation}
where $P$ is the orthogonal projection of $L^2(e^{-2\phi})$ onto
$\mathcal{F}^2_\phi$. In general, $\bar z f\in L^2(e^{-2\phi})$  does not hold
for all $f\in\mathcal{F}^2_\phi$ (see e.g.\cite{schneiderknirsch}), but it
follows from Theorem  \ref{kernelestimate} that  ${\bar z} k_\lambda\in
L^2(e^{-2\phi})$ for all $\lambda\in \C$. Since the subset $\hbox{Span}
\{k_\lambda: \lambda\in\C\}$ is dense in $\mathcal{F}^2_\phi$, we deduce from
(\ref{hankeldbar}) that $N$ coincides with the big Hankel operator acting on
$\mathcal{F}^2_\phi$ with symbol $\bar z$. Motivated by this last fact, we now
aim to study Hankel operators with anti-analytic symbols on
$\mathcal{F}^2_\phi$. Given an entire function $g$ so that there exists a dense
subset $A$ of $\mathcal{F}^2_\phi$ with $\bar g f\in L^2(e^{-2\phi})$ for $f\in
A$, the big Hankel operator with symbol $\bar g$ is densely defined by 
$$
H_{\bar g}f= \bar g f-P(\bar g f)=(I-P)(\bar g f),\quad f\in A,
$$ 
where $P$ is the orthogonal projection of $L^2(e^{-2\phi})$ onto 
$\mathcal{F}^2_\phi$. 
We consider symbols $g$ such that 
$$\bar{g} k_\lambda \in L^2(e^{-2\phi}) \hbox{ for all }\lambda\in \C.$$
It follows from Theorem \ref{kernelestimate} that, for example, polynomial 
symbols satisfy this assumption. By the reproducing
formula in $\mathcal{F}^2_\phi$ we get
\begin{equation}\label{kernelaction}
H_{\bar g} k_\lambda (z)=(\overline{g(z)}- \overline{g(\lambda)}) k_\lambda(z),
\quad z,\lambda\in\C.
\end{equation}
For the sake of completeness we shall first characterize the boundedness and 
compactness of $H_{\bar g}$.
Let us state the following theorem due to H\"ormander which is essential to our 
approach.
\begin{Theorem}\label{hormander}\cite{hoermander} Let $\Omega\subseteq \C$ be a 
domain and $\phi\in C^2(\Omega)$
be such that $\Delta\phi \ge 0$. For any $f\in L^2_{loc}(\Omega)$ there exists 
a solution $u$ to $\bar\partial u=f$
such that 
$$
\int |u|^2 e^{-2\phi} dm \le \int \frac{|f|^2}{\Delta \phi} e^{-2\phi} dm.
$$
\end{Theorem}

\begin{Theorem}\label{boundedness}
$H_{\bar g}$ extends to a bounded linear operator on $\mathcal{F}^2_\phi$ if 
and only if $|g'| \rho$ is bounded.
\end{Theorem}

\begin{proof}
Assume first that $|g'|\rho$ is bounded. Then notice that for $f\in \hbox{Span} 
\{k_\lambda: \lambda\in\C\}$, $H_{\bar g}f$ is the solution
to $\bar\partial u=\bar g'f $ of minimal $L^2(e^{-2\phi})$-norm. By Theorem 
\ref{hormander} 
and Proposition \ref{regularization} we have
\begin{equation}\label{he}
\int_\C |H_{\bar g} f|^2 e^{-2\phi} dm \lesssim \int_\C |f|^2 |g'|^2 \rho^2 
dm\le (\sup |g'|\rho)^2  \|f\|^2,
\end{equation} 
 which shows that $H_{\bar g}$ can be extended to a bounded linear operator  on
$\mathcal{F}^2_\phi$.

Conversely, assume that $H_{\bar g}$ is bounded. Then we have
$\|H_{\bar g} k_\lambda\|<M$ for $\lambda\in \C$, and using  relation
(\ref{kernelaction}) together with Lemmas \ref{close} and \ref{comparable} we
obtain 
\begin{eqnarray}
M>\|H_{\bar g} k_\lambda \|^2&=&\int_\C |g(z)-g(\lambda)|^2 |k_\lambda(z)|^2
e^{-2\phi(z)} dm(z)\nonumber \\
&\ge& \int_{|z-\lambda|<\alpha \rho(\lambda)} |g(z)-g(\lambda)|^2
|k_\lambda(z)|^2 e^{-2\phi(z)} dm(z)\nonumber \\
&\gtrsim & \frac{1}{\rho^2(\lambda)} \int_{|z-\lambda|<\alpha \rho(\lambda)}
|g(z)-g(\lambda)|^2  dm(z)\nonumber,
\end{eqnarray}
for $\alpha$ small enough.
By the subharmonicity of $|g|$ and the Cauchy formula applied to
$g_\lambda(z)=g(z)- g(\lambda)$ we can now conclude
$$
|g'(\lambda)\rho(\lambda)|\lesssim  \frac{1}{\rho^2(\lambda)}
\int_{|z-\lambda|<\alpha \rho(\lambda)} |g(z)-g(\lambda)|^2  dm(z)<M, \quad
\lambda\in\C.
$$
\end{proof}

\noindent {\bf Remark.} The fact that $\rho$ can have at most polynomial decay 
(see relation (\ref{decayrho})) implies that
$H_{\bar g}$ is bounded only for polynomial symbols of degree smaller than the
order of decay of $\rho$.  Notice also that if $H_{\bar g}$ is bounded, then
$\rho$ has to be bounded,
since $g$ is a polynomial. 
\bigskip

\begin{Theorem}\label{compactness}  $H_{\bar g}$ 
is compact if and only if $|g'(\lambda)| \rho(\lambda)\rightarrow 0$ as 
$|\lambda|\rightarrow\infty$.
\end{Theorem}

\begin{proof}
Assume first that $|g'(\lambda)| \rho(\lambda)\rightarrow 0$ as
$|\lambda|\rightarrow\infty$. 
As  in relation (\ref{he}) we have
$$
\|H_{\bar g}f\|^2\le \int_\C |g'|^2\rho^2 |f|^2 e^{-2\phi} dm=\|M_{g'\rho}f\|^2,
$$ where $M_{g'\rho}:\mathcal{F}^2_\phi\rightarrow L^2(e^{-2\phi})$ is given by
$M_{g'\rho} f=g'\rho f$.  Hence, if $M_{g'\rho}$ is compact, then $H_{\bar g}$
is compact. We first show that,  for $R>0$, the truncation of $M_{g'\rho}$ given
by
$$
M^R_{g'\rho}f=\chi_{\{|z|<R\}} \, g'\rho f
$$
is compact.
To this end, let $\{f_n\}$ be a bounded sequence in $\mathcal{F}^2_\phi$, i.e. 
$\|f_n\|<M$. Since pointwise evaluation is bounded, 
we deduce that $\{f_n\}$ is a 
normal family and it therefore contains a subsequence $\{f_{n_k}\}$ uniformly 
convergent on compacts to
an entire function $f$. By Fatou's lemma we obtain $f\in \mathcal{F}^2_\phi$. 
Then $f_{n_k}-f\rightarrow 0$
uniformly on compacts
and $\|f_n-f\|<2M$. Hence in order to show that $M^R_{g'\rho}$ is compact,
it is enough to show that for any sequence
$f_n$ (by abuse of notation) that is bounded in the norm and converges uniformly
 to zero on compact sets,
we have $\|M^R_{g'\rho} f_n\|\rightarrow 0$ as $n\rightarrow\infty$.

But this is quite easy to see, as
$$
\|M^R_{g'\rho} f_n\|^2\le \sup_{|z|<R} |f_n|^2 \int_{|z|<R} |g'|^2 \rho^2 
e^{-2\phi}dm \rightarrow 0,
$$
as $n\rightarrow\infty$.
Now 
$$
\|(M_{g'\rho}-M^R_{g'\rho}) f\|^2=\int_{|z|\ge R}|g'|^2 \rho^2  |f|^2
e^{-2\phi}dm\le 
\sup_{|z|>R}|g'|^2 \rho^2  \int_{\C} |f|^2  e^{-2\phi}dm, \ \ 
f\in\mathcal{F}^2_\phi,
$$
which shows that  $\|M_{g'\rho}-M^R_{g'\rho}\|\rightarrow 0$ as  $R\rightarrow
\infty$,  and therefore $M_{g'\rho}$ is compact,
and consequently $H_{\bar g}$ is compact.

Suppose  now $H_{\bar g}$ is compact.  The set $\{k_\lambda\}_{\lambda\in\C}$ is
 bounded in $\mathcal{F}^2_\phi$. By compactness
it follows that the set $\{H_{\bar g}k_\lambda\}_{\lambda\in\C}$ is relatively 
compact in $L^2(e^{-2\phi})$.
Then by the Riesz-Tamarkin compactness theorem (see \cite{berger}) we have 
\begin{equation}
\lim_{R\rightarrow\infty} \int_{|z|>R} |H_{\bar g}k_\lambda|^2 e^{-2\phi} dm =0,
\end{equation}
uniformly in $\lambda$.
Since $H_{\bar g}$ is bounded, we have $B:=\sup_\zeta \rho(\zeta)<\infty$.
For $|\lambda|>R+B$, the inclusion $\{|z-\lambda|\le \rho(\lambda)\}\subset
\{|z|>R\}$ holds, and then for $\alpha>0$ sufficiently small we have
by Lemma \ref{close}
\begin{eqnarray*}
 \int_{|z|>R} |H_{\bar g}k_\lambda|^2 e^{-2\phi} dm&=&\int_{|z|>R} 
|g(z)-g(\lambda)|^2 |k_\lambda(z)|^2 e^{-2\phi(z)} dm(z)\\
&\gtrsim& \int_{|z-\lambda|< \alpha \rho(\lambda)} |g(z)-g(\lambda)|^2 
|k_\lambda(z)|^2 e^{-2\phi(z)} dm(z)\\
&\gtrsim& \frac{1}{\rho^2(\lambda)}\int_{|z-\lambda|< \alpha \rho(\lambda)} 
|g(z)-g(\lambda)|^2 dm(z)\\
&\gtrsim& \rho^2(\lambda) |g'(\lambda)|^2,
\end{eqnarray*}
where the last step above follows again by the Cauchy formula and  the
subharmonicity of $|g|$. This shows that
$$
\lim_{|\lambda|\rightarrow \infty} |g'(\lambda)|\rho(\lambda)=0.
$$
\end{proof}

In the study of the Schatten class membership of $H_{\bar g}$ we use  the
following well-known inequality:
If $T$ is a compact operator from $\mathcal{F}^2_\phi$ to a Hilbert  space
$\mathcal{H}$,  we have
\begin{equation}\label{spineq}
\int_\C \|T k_\lambda\|^p \frac{dm(\lambda)}{\rho^2(\lambda)}\lesssim
\|T\|^p_{\mathcal{S}^p},
\end{equation}
for $p\ge 2$.
To see this, let 
$$
T=\sum_n \lambda_n \langle \cdot, e_n\rangle f_n,
$$
be the canonical form of $T$, where $(e_n)$ is an  orthonormal basis  in
$\mathcal{F}^2_\phi$,
$(f_n)$ is an orthonormal set in $H$, and the $\lambda_n$'s
are the singular numbers of $T$.
Then 
$$
TK(\cdot, \lambda)
= \sum_n \lambda_n \overline{e_n(\lambda)} f_n, \quad \lambda\in\C.
$$
From this we deduce
$$
\int_\C \|TK(\cdot, \lambda)\|^2 e^{-2\phi(\lambda)} dm(\lambda)=\int _\C \sum_n
\lambda_n^2 |e_n(\lambda)|^2 e^{-2\phi(\lambda)} dm(\lambda)=\sum_n \lambda_n^2.
$$
Hence
$$
\int_\C \|T k_\lambda\|^2 \frac{dm(\lambda)}{\rho^2(\lambda)}\sim\int_\C 
\|TK(\cdot, \lambda)\|^2 e^{-2\phi(\lambda)}
dm(\lambda)=\|T\|^2_{\mathcal{S}^2}.
$$
For $p=\infty$, we have
$$
\sup_\lambda \|Tk_\lambda\|\le \|T\|_{\mathcal{S}^\infty}.
$$
Then (\ref{spineq}) follows by interpolation.
\bigskip

\begin{Theorem}\label{schattencl}
Suppose $H_{\bar g}$ is bounded. Then $H_{\bar g}\in\mathcal{S}^p$ with $p> 2$ 
if and only if 
$g'\rho\in L^p(1/\rho^2)$.  Moreover, $H_{\bar g}$ fails to be Hilbert-Schmidt, 
unless $g$ is constant.
\end{Theorem}

\begin{proof}
Suppose $H_{\bar g}\in \mathcal{S}^p$ with $p\ge 2$. Then by (\ref{spineq})  and
using arguments similar to those above we
have
\begin{eqnarray*}
\infty>\int_\C \|H_{\bar g} k_\lambda \|^p
\frac{dm(\lambda)}{\rho^2(\lambda)}&=&\int_\C \Bigl(\int_\C |g(z)-g(\lambda)|^2
|k_\lambda(z)|^2 e^{-2\phi(z)}\,dm(z)\Bigr)^{p/2}
 \frac{dm(\lambda)}{\rho^2(\lambda)}\\
&\gtrsim& 
\int_\C \Bigl(\int_{|z-\lambda|<\alpha \rho(\lambda)} |g(z)-g(\lambda)|^2
|k_\lambda(z)|^2 e^{-2\phi(z)}\,dm(z)\Bigr)^{p/2}
 \frac{dm(\lambda)}{\rho^2(\lambda)}\\
&\gtrsim& \int_\C |g'(\lambda)\rho(\lambda)|^p 
\frac{dm(\lambda)}{\rho^2(\lambda)},
\end{eqnarray*}
for $\alpha$ small enough.
With this the necessity is proven.
In particular,  the above relation shows that $H_{\bar g}$ cannot be 
Hilbert-Schmidt for nonconstant 
anti-analytic symbols.

To prove the sufficiency, assume $g'\rho\in L^p(1/\rho^2)$. Then a
subharmonicity argument
shows that $|g'(\lambda)|\rho(\lambda)\rightarrow 0$ as
$|\lambda|\rightarrow\infty$. 
As in the proof  of Theorem \ref{compactness} we have
$$
\|H_{\bar g}f\|\lesssim \|M_{g'\rho}f\|,\quad f\in \mathcal{F}^2_\phi.
$$  
Therefore $M_{g'\rho}\in \mathcal{S}^p$ for some $p>2$, implies  $H_{\bar g}\in
\mathcal{S}^p$. Indeed, this follows from the criterion (see \cite{gohberg}):  
A linear operator $S:H_1\rightarrow H_2$, where $H_1,H_2$
are separable Hilbert spaces, belongs to ${\mathcal{S}^p},  p\ge 2$, if and only
if 
$\sum \|S e_n\|^p<\infty$, for any orthonormal basis $\{e_n\}$ of $H_1$.   
 We notice that for $f,h\in \mathcal{F}^2_\phi$ we have
\begin{equation*}
\langle M_{g'\rho}^*M_{g'\rho} f , h \rangle=\langle M_{g'\rho} f ,  M_{g'\rho}
h \rangle= \int_\C  f \bar h |g'|^2\rho^2 e^{-2\phi} \, dm=
\langle T_{|g'|^2 \rho^2} f, h\rangle,
\end{equation*}   
where $T_{|g'|^2 \rho^2}$ is the Toeplitz operator on  $\mathcal{F}^2_\phi$ with
symbol $|g'|^2\rho^2$.
In order to show that $M_{b'\rho}\in\mathcal{S}^p$,  we are going to prove that 
$T_{|g'|^2\rho^2}=M_{g'\rho}^*M_{g'\rho}
\in \mathcal{S}^{p/2}$. 
Since  $|g'(\lambda)|\rho(\lambda)\rightarrow 0$ as
$|\lambda|\rightarrow\infty$, the proof of the sufficiency in Theorem 
\ref{compactness} shows that $M_{g'\rho}$ is compact, and hence  $T_{|g'|^2
\rho^2}$ is compact.
Denote $G=|g'|^2\rho^2$ for convenience.  The operator $T_G$  is also positive
and self-adjoint, and
it is then given by
\begin{equation*}
T_G=\sum_n \lambda_n \langle \cdot, e_n\rangle e_n,
\end{equation*}
where $\lambda_n$ are the singular numbers of $T_G$, and $e_n$ is  an
orthonormal basis in $\mathcal{F}^2_\phi$.  
Then 
$$
\lambda_n=\langle T_G e_n, e_n\rangle= \int_\C |e_n|^2 G e^{-2\phi} \, dm,
$$
and by Jensen's inequality we get
$$
\lambda_n^{p/2}\le  \int_\C  G^{p/2} |e_n|^2 e^{-2\phi} \, dm
$$
using the fact that $|e_n|^2 e^{-2\phi} \, dm$ is a probability measure on $\C$.
Taking into account the fact that $K(z,\zeta)=\sum e_n(z)
\overline{e_n(\zeta)}$,  we can sum up over $n$ in 
the previous relation to deduce
\begin{eqnarray*}
\sum_n \lambda_n^{p/2}&\le& \sum_n  \int_\C  G^{p/2} |e_n|^2 e^{-2\phi} \, dm\\
&=&  \int_\C  G(z)^{p/2} K(z,z)e^{-2\phi(z)} \, dm(z)\\
&\lesssim& \int_\C  G(z)^{p/2} \frac{1}{\rho^2(z)} \, dm(z)<\infty, 
\end{eqnarray*}
by our assumption. Thus $T_G\in \mathcal{S}^{p/2}$, and consequently  $H_{\bar
g}\in \mathcal{S}^p$.
\end{proof}

\section {The canonical solution to $\bar\partial$ on $L^2(e^{-2\phi})$}

\noindent For $g=z$ in Theorem \ref{schattencl} we obtain that the restriction
of the canonical solution operator $N$ to $\bar\partial$ to the generalized Fock
space $\mathcal{F}^2_\phi$ is  never Hilbert-Schmidt and it belongs to
$\mathcal{S}^p$ for $p>2$ if and only if 
\begin{equation}\label{schattenco}
\int_\C \rho^{p-2}(z)\, dm(z)<\infty.
\end{equation}
The aim of this section is to show that the condition  above is sufficient for
$N$ to belong to $\mathcal{S}^p$, even when 
defined on the whole of $L^2(e^{-2\phi})$.

\noindent For the integral kernel $C(z,\zeta)$ of $N$, i.e.
$$
Nf(z)=\int_\C e^{\phi(z)-\phi(\zeta)} C(z,\zeta) f(\zeta)\, dm(\zeta),  \quad
f\in L^2(e^{-2\phi}),
$$  
the following estimates were obtained in \cite{jj}
\begin{Theorem}\label{estimate1}\cite{jj}
There exists $\varepsilon>0$ such that
$$
|C(z,\zeta)|\lesssim
\begin{cases}
|z-\zeta|^{-1},& |z-\zeta|\le\rho(z),\\
\rho^{-1}(z)\exp(-d_\phi(z,\zeta)^\varepsilon),& |z-\zeta|\ge\rho(z).
\end{cases}
$$
\end{Theorem}
\noindent To prove our main result we use these estimates together with a
criterion for an integral operator to belong to Schatten classes for $p\ge 2$
obtained in \cite{russo}. Given a measure space $(X,\mu)$, let $G(x,y)$ be a
complex-valued measurable function on $X\times X$ and denote
$G^*(x,y)=\overline{G(y,x)}$.  Consider the mixed normed space 
$$
L^p(L^q)=\Bigl\{G:\   \int \Bigl(\int |G(x,y)|^q  d\mu(y)\Bigr)^{p/q} d\mu
(x)<\infty\Bigr\}
$$
\begin{Theorem}\label{sufficientc} \cite{russo}
Let $p\ge 2$ and let $(X,\mu)$ be as above. If  $G,G^*\in L^p(L^{p'})$, where $1
/p +1 /p' = 1$, then
the integral operator with kernel $G(x,y)$ given by
$$
T f (x)=\int G(x,y) f(y) d\mu(y), \quad f\in L^2(d\mu),
$$  
belongs to ${\mathcal S}^p$.
\end{Theorem}
\noindent A first version of the above theorem  was proven in \cite{russo} (see
also \cite{wolff}) and
subsequently improved in \cite{peetre}, where sharper conditions on the kernel 
$G$ were given. 

\begin{Theorem}\label{maintheorem}
The operator $N$ is never Hilbert-Schmidt. For $p>2$, $N$ belongs
to the Schatten class $\mathcal{S}^p$ if and only if (\ref{schattenco}) holds.
\end{Theorem}

\begin{proof} The necessity follows from Theorem \ref{schattencl}.
It remains to prove the sufficiency.   Assume $\rho$ satisfies
(\ref{schattenco}) for some $p>2$.  In order to
prove that $N\in \mathcal{S}^p$, we want  apply Theorem \ref{sufficientc}.  To
this end consider the
unitary operator $U:L^2 \rightarrow L^2(e^{-2\phi})$ given by 
$$
Uf=fe^{\phi}.
$$
Then $N\in \mathcal{S}^p$ if and only if $U^*NU\in\mathcal{S}^p$. Notice that 
$$
U^*NU f(z)=\int_\C C(z,\zeta) f(\zeta)\, dm(\zeta), \quad f\in L^2.
$$
Now it is enough to show that the kernel $C(z,\zeta)$ of $U^*NU$ satisfies
the
conditions in Theorem  \ref{sufficientc},  and then the conclusion will easily
follow.
We shall first estimate 
\begin{eqnarray}\label{term1}
\|C\|^p_{L^p(L^{p'})}= \int_{\C} \Bigl(\int_\C |C(z,\zeta)|^{p'}\, 
dm(\zeta)\Bigr)^{p/p'}\, dm(z).
\end{eqnarray}
Theorem \ref{estimate1} implies
\begin{eqnarray}\label{r1}
\int_{\C} |C(z,\zeta)|^{p'}\, dm(\zeta)&\lesssim&  \int_{|z-\zeta |\le
\rho(z)} \frac{dm(\zeta)}{|z-\zeta|^{p'}} +
\int_{|z-\zeta|> \rho(z)} \frac{dm(\zeta)}{\rho(z)^{p'} \exp
(p'd_\phi^{\varepsilon}(z,\zeta))}\nonumber\\
&\lesssim& \rho(z)^{2-p'}+
\int_{|z-\zeta|> \rho(z)} \frac{dm(\zeta)}{\rho(z)^{p'} \exp
d_\phi^{\varepsilon_1}(z,\zeta)},
\end{eqnarray}
for $0<\varepsilon_1<\varepsilon$. Now for $|z-\zeta| \le \rho(z)$ or
$|z-\zeta|\le \rho(\zeta)$ we have
$\rho(z)\sim\rho(\zeta)$ by Lemma \ref{comparable}.  On the other hand, for
$(z,\zeta)\in\{|z-\zeta|> \rho(z)\}\cap\{|z-\zeta|> \rho(\zeta)\}$, Lemmas
\ref{distance1}-\ref{distance2} imply
$$
\frac{\rho(\zeta)^2}{\exp d_\phi^{\varepsilon_1}(z,\zeta)} \lesssim
\frac{\rho(z)^2}{\exp d_\phi^{\varepsilon_2}(z,\zeta)},
$$  
for some $\varepsilon_2>0$.  Using this in (\ref{r1}) we get
\begin{eqnarray}
\int_{\C} |C(z,\zeta)|^{p'}\,
dm(\zeta)&\lesssim&\rho(z)^{2-p'}+\rho(z)^{2-p'}
\int_{|z-\zeta|> \rho(z)} \frac{1}{\exp d_\phi^{\varepsilon_2}(z,\zeta)}
\frac{dm(\zeta)}{\rho^2(\zeta)}\nonumber\\
&\lesssim& \rho(z)^{2-p'},\nonumber
\end{eqnarray}
where the last step above follows by Proposition  \ref{regularization} and 
Lemma \ref{estimate}.
Returning to (\ref{term1}) we obtain
$$
\|C\|^p_{L^p(L^{p'})}= \int_\C \rho(z)^{(2-p')p/p'}  dm(z)=\int_\C \rho(z)^{p-2}
dm(z)<\infty,
$$
by our assumption.
It remains to show that $\|C^*\|_{L^p(L^{p'})}<\infty$.  Although the estimates
are analogous
in this case, we include them for the sake of completeness. We have
\begin{eqnarray}\label{term2}
\|C^*\|^p_{L^p(L^{p'})}=\int_{\C} \Bigl(\int_\C  |C(z,\zeta)|^{p'}\,
dm(z)\Bigr)^{p/p'}\, dm(\zeta).
\end{eqnarray}
As before, by Theorem \ref{estimate1} and Lemma \ref{comparable} we get
\begin{eqnarray}
\int_{\C} |C(z,\zeta)|^{p'}\, dm(z)&\lesssim& \int_{|z-\zeta |\le \rho(z)}
\frac{dm(z)}{|z-\zeta|^{p'}} + \int_{|z-\zeta|> \rho(z)}
\frac{dm(z)}{\rho(z)^{p'}\exp (p'd_\phi^{\varepsilon}(z,\zeta))}\nonumber\\
&\lesssim& \int_{|z-\zeta |\le c\rho(\zeta)} \frac{dm(z)}{|z-\zeta|^{p'}}
+ \int_{|z-\zeta|> \rho(z)} \frac{dm(z)}{\rho(z)^{p'}\exp
(p'd_\phi^{\varepsilon}(z,\zeta))}\nonumber\\ &\lesssim&
\rho(\zeta)^{2-p'}\Bigl(1+ \int_{|z-\zeta|> \rho(z)} \frac{1}{\exp
d_\phi^{\varepsilon_1}(z,\zeta)} \frac{dm(z)}{\rho^2(\zeta)}\Bigr) \nonumber,
\end{eqnarray}
where $c>0$, and the last step above follows by Lemmas  
\ref{distance1}-\ref{distance2}. By 
Proposition \ref{regularization} and Lemma \ref{estimate} we obtain
$$
\int_{\C} |C(z,\zeta)|^{p'}\, dm(z)\lesssim \rho(\zeta)^{2-p'},
$$
and hence by (\ref{term2}) we get
\begin{equation*}
\|C^*\|^p_{L^p(L^{p'})}\lesssim\int_{\C} |\rho(\zeta)|^{p-2}\, dm(\zeta)<\infty.
\end{equation*}
With this the proof is complete.
\end{proof}


\begin{thebibliography}{99}

\bibitem{peetre}
 J. Arazy, S. D. Fisher, S.  Janson, J. Peetre, Membership of Hankel operators
on the ball in unitary ideals,
{\it  J. London Math. Soc.} {\bf 43} (1991),  485-508. 


\bibitem{bergman}
 S. Bergman, {\it The Kernel Function and Conformal Mapping}, 2nd ed., A.M.S.
Survey V, Providence, 1970. 


\bibitem{berger}
 M. S. Berger, {\it Nonlinearity and functional analysis},  Academic Press, New
York-London, 1977.


\bibitem{youssfi}
 H. Bommier-Hato, E. H. Youssfi, Hankel operators on weighted Fock spaces, {\it
Integral Equations Operator Theory} {\bf 59} (2007), 1-17.

\bibitem{christ}
 M. Christ, On the $\overline\partial$ equation in weighted $L\sp 2$ norms in
$C\sp 1$,{\it J. Geom. Anal.\,}{\bf 1}(1991),  193-230.

\bibitem{fustraube} 
S. Fu, E. J. Straube,  Compactness in the $\overline\partial$-Neumann problem,
Complex analysis and geometry (Columbus, OH, 1999), 141-160, Ohio State Univ.
Math. Res. Inst. Publ., 9, de Gruyter, Berlin, 2001. 

\bibitem{gohberg}
I. C. Gohberg and M. G. Krein,  Introduction to the theory of linear
nonselfadjoint operators.  Translations of Mathematical Monographs, Amer. Math.
Soc., Providence, R.I., 1969.

\bibitem{haslinger1} F. Haslinger, The canonical solution operator to
$\overline\partial$ restricted to spaces of entire functions, {\it Ann. Fac.
Sci. Toulouse Math.} (6) {\bf 11} (2002),  57-70.

\bibitem{haslhelffer}F. Haslinger and B. Helffer, 
Compactness of the solution operator to $\overline\partial$ in weighted
$L^2$-spaces,
{\it J. Funct. Anal.}  {\bf 243} (2007), 679-697. 

\bibitem{haslingerlamel} 
F. Haslinger and B. Lamel,  Spectral properties of the canonical solution
operator to $\overline\partial$, {\it J. Funct. Anal.} {\bf 255} (2008), 13-24. 

\bibitem{hoermander}
L.  H\"ormander, $L^{2}$ estimates and existence theorems for the $\bar
\partial $ operator, {\it Acta Math. } {\bf 113} (1965),  89-152.

\bibitem{wolff}
S. Janson, T.  H.  Wolff, 
Schatten classes and commutators of singular integral operators,
{\it Ark. Mat.} {\bf 20} (1982), 301-310. 

\bibitem{schneiderknirsch}
W. Knirsch, G. Schneider, Continuity and Schatten-von Neumann $p$-class
membership of Hankel operators with anti-holomorphic symbols on (generalized)
Fock spaces, {\it J. Math. Anal. Appl.} {\bf 320} (2006), 403-414.

\bibitem{linrochberg}
P. Lin, R. Rochberg,  Hankel operators on the weighted Bergman spaces with
exponential type weights,
{\it Integral Equations Operator Theory} {\bf 21} (1995), 460-483. 

\bibitem{linrochberg1}
P. Lin, R. Rochberg,  Trace ideal criteria for Toeplitz and Hankel operators on
the weighted Bergman spaces with exponential type weights, {\it Pacific J.
Math.} {\bf 173} (1996), 127-146. 

\bibitem{loverayoussfi}
S. Lovera, E. H. Youssfi, 
Spectral properties of the $\overline\partial$-canonical solution operator, 
{\it J. Funct. Anal.} {\bf 208} (2004), 360-376. 

\bibitem{mmoc}
N. Marco, X. Massaneda and J. Ortega-Cerd\`a,  Interpolating and sampling
sequences for entire functions, {\it Geom. Funct. Anal.} {\bf 13} (2003), 
862-914.

\bibitem{jj} J. Marzo and J. Ortega-Cerd\`a, Pointwise estimates for the
Bergmam kernel of the weighted Fock space, {\it J. Geom. Anal.}, 
{\bf 19} (2009), 890-910.

\bibitem{russo}
B. Russo , On the Hausdorff-Young theorem for integral operators, {\it Pacific
J. Math.} {\bf 68} (1977),  241-253.

\bibitem{schneider}
G. Schneider, A note on Schatten-class membership of Hankel operators with
anti-holomorphic symbols on generalized Fock-spaces, {\it Math. Nachr.} {\bf
282} (2009), 99-103. 



\end{thebibliography}
\end{document}